\documentclass[a4paper]{amsart}
\usepackage{amsmath}
\usepackage{amsfonts}
\usepackage{amssymb}
\usepackage{amsthm}
\usepackage{calc}
\usepackage{times}
\usepackage{epsf}
\usepackage{epsfig}
\usepackage{footnpag}
\usepackage{wrapfig}
\usepackage{floatflt}
\usepackage{graphicx}

\title{A class of non--associated materials: $n$--monotone materials}  
\author{C. Vall\'{e}e, C. Lerintiu, J. Chaoufi, D. Fortun\'{e}, M. Ban, K. Atchonouglo}

\theoremstyle{definition}         
\newtheorem{definition}{Definition}[section]

\newtheorem{remark}[definition]{Remark}
\newtheorem{example}[definition]{Example}

\theoremstyle{plain}              
\newtheorem{theorem}[definition]{Theorem}
\newtheorem{propos}[definition]{Proposition}
\newtheorem{lemma}[definition]{Lemma}

\numberwithin{equation}{section}
\address{Universit\'{e} de Poitiers \\ Institut Pprime\\ SP2MI\\ UPR CNRS 3346\\ Bd Marie et Pierre Curie\\ T\'{e}l\'{e}port 2\\ BP 30179\\ 86962 Futuroscope Chasseneuil Cedex\\ France} 
              \address{Tel.: +33-683746231}
              \address{E-mail adress: claude.vallee@univ-poitiers.fr} 
\address{Universit\'{e} de Poitiers \\ Institut Pprime\\ SP2MI\\ UPR CNRS 3346\\ Bd Marie et Pierre Curie\\ T\'{e}l\'{e}port 2\\ BP 30179\\ 86962 Futuroscope Chasseneuil Cedex\\ France} 
\address{Universit\'{e} Ibn Zohr\\ Facult\'{e} des Sciences\\ cit\'{e} Dakhla \\B.P. 8106\\ 80000 Agadir\\ Morocco} 
\address{21\\ rue du Hameau du Cherpe\\ 86280 Saint Beno\^{i}t\\ France} 
\address{Institut f\"ur Allgemeine Mechanik\\ RWTH Aachen\\ Templergraben 64\\ 52056 Aachen\\ Germany} 
\address{Universit\'{e} de Lom\'{e}, Facult\'{e} des Sciences, B.P. 1515, Lom\'{e}, Togo}
\keywords{non--associated constitutive laws; elastic materials;  $n$--cyclically monotone operators;
Fitzpatrick's sequences; bipotentials} 
\subjclass[2000]{Primary: 74D10 \and 47H05 \and Secondary: 47H04} 

\begin{document}
\begin{abstract}
Generalized Standard Materials are governed by maximal cyclically monotone operators 
and modeled by convex potentials. G\'{e}ry de Saxc\'{e}'s Implicit Standard Materials are modeled 
by biconvex bipotentials. We analyze the intermediate class of $n$--monotone materials governed by 
maximal $n$--monotone operators and modeled by Fitzpatrick's functions. Revisiting the model of elastic 
materials initiated by Robert Hooke, and insisting on the linearity, coaxiality and monotonicity properties 
of the constitutive law, we illustrate that Fitzpatrick's representation of $n$--monotone operators coming 
from convex analysis provides a constructive method to discover the best bipotential modeling a 
$n$--monotone material. Giving up the symmetry of the linear constitutive laws, we find out that 
$n$--monotonicity is a relevant criterion for the materials characterization and classification.
%
%
%
%
\end{abstract}

\maketitle
%
%
%
%
\section{Introduction}
\label{sec:1}

Standard Materials are modeled by differentiable potentials. Mainly to capture 
set--valued constitutive laws, for example plastic flow rules (\cite{deb76,mor76})
they were extended to the so--called "Generalized Standard 
Materials" (GSM) modeled by lower semi--continuous (lsc) convex potentials (\cite{hac97,hal75}). 
However, this extension failed to describe Coulomb's dry friction law. 

In 1991, considering such an implicit constitutive law,  G\'{e}ry de Saxc\'{e} and 
Zhi--Qiang Feng (\cite{sax91,sax92,sax98}) proposed a new generalization 
which they called "Implicit Standard Material" (ISM). This new class of materials is 
modeled by a point--to--point function which they called a bipotential. In the particular 
case of a GSM, the bipotential becomes simply the sum of the potential and its conjugate.
 
For a given GSM, a theorem due to R. Tyrrel Rockafellar (\cite{roc170,roc66}) and 
Jean--Jacques Moreau (\cite{mor03}) provides a constructive method to retrieve the potential 
from its subdifferential. A similar, crucial question for a given ISM is: how to 
retrieve the bipotential from the implicit constitutive law?

Independently, in 1988, Simon Fitzpatrick (\cite{fit88}), in order to simplify 
the study of monotone operators (\cite{bau307,mar05}), made a proposal to replace these 
multifunctions by point--to--point functions, nowadays called Fitzpatrick's functions.
 
Is there a relation between G\'{e}ry de Saxc\'{e}'s bipotentials representing 
ISM constitutive laws and Fitzpatrick's functions representing maximal 
monotone multifunctions? Can this last representation coming from convex 
analysis provide a method to construct the best bipotential 
modeling a given ISM? The aim of our paper is to discover the largest class 
of ISM for which the answers to both questions are positive.
 
Revisiting the modeling of elastic materials initiated by Robert Hooke, we 
illustrate Fitzpatrick's method on the example of linear coaxial constitutive laws. 

%
%
%
%
%
\section{Standard Materials}
\label{sec:2}

A constitutive law relating a strain--like variable $x$ belonging to a real
Banach space $X$ (with norm $\|\cdot\|$) and a stress--like variable $y$ 
belonging to the topological dual space $Y=X^*$ (with the duality product $\langle\cdot,\cdot\rangle$),
is a subset of the product set $X\times Y$. This subset can be regarded as the graph 
\begin{equation}
G(T)=\left\{(x,y)\in X\times Y \mid y\in Tx\right\} \nonumber
\end{equation}
of a multivalued operator $T:X\longrightarrow 2^Y$. In finite dimension, when this subset is a maximal lagrangian 
submanifold of the linear space $X\times Y$ (made symplectic by the canonical Darboux 2--form), there 
exists a differentiable function $\phi$, called "potential", such that the constitutive law reads
\begin{equation}
y=\mathrm{grad}\, \phi(x)  \mbox{. }
\nonumber
\end{equation}

If, additionally, this potential is convex, the inverse constitutive law reads
\begin{equation}
x=\mathrm{grad}\, \phi^*(y) \nonumber
\end{equation}
with $\phi^*$ the Legendre transform of the potential $\phi$ (called 
"conjugate potential"). A material whose behavior can be described by a differentiable 
potential is referred to as a "Standard Material" (SM).

%
%
%
%
%
\section{Generalized Standard Materials}
\label{sec:3}

For many materials, the relation between $x$ and $y$ is multivalued. Giving up 
the differentiability of the potential $\phi$, but retaining its convexity and its 
lower semi--continuity, a large class of materials, called "Generalized Standard 
Materials" (GSM), can be described by one of the following three equivalent constitutive laws:
\begin{enumerate}
\item[$(i)$] $y\in\partial\phi (x)$
\bigskip

\item[$(ii)$] $x\in\partial\phi^*(y)$
\bigskip

\item[$(iii)$] $\phi(x)+\phi^*(y)=\langle x,y\rangle \mbox{.}$
\end{enumerate}

\begin{remark}
The convexity of the potential $\phi$ allows to express the conjugate potential
$\phi^*$ as a supremum (\cite{mor03})
\begin{equation}
\phi^*(y)=\sup_{x\in X}\left[\langle x,y\rangle -\phi(x)\right]   \mbox{. }
\nonumber
\end{equation}
\end{remark}

\begin{remark}
The subdifferentials 
\begin{equation}
\partial\phi(x) = \left\{y\in Y \mid \forall\xi\in X, \phi(\xi)\geq\phi(x)+\langle \xi-x,y\rangle\right\} 
\nonumber
\end{equation}
and
\begin{equation}
\partial\phi^*(y) = \left\{x\in X \mid \forall\eta\in Y, \phi^*(\eta)\geq\phi^*(y)+\langle x,\eta-y\rangle\right\}
\nonumber
\end{equation}
generalize (\cite{mor03}) the gradients of the potentials $\phi$ and $\phi^*$ when 
these convex potentials are not differentiable. Their elements are called subgradients. 
At a point $x$ (respectively $y$) where $\phi$ (respectively $\phi^*$) is both convex 
and differentiable, the set of subgradients $\partial\phi(x)$ (respectively $\partial\phi^*(y)$) 
reduces to the unique gradient at $x$ (respectively at $y$). 
\end{remark}

%
%
%
%
%
\section{Implicit Standard Materials}
\label{sec:4}

%
%
\subsection{Bipotentials}

The equality $(iii)$ of Section \ref{sec:3} can be regarded as an extremal case of Fenchel's inequality
\begin{equation}
\phi(x)+\phi^*(y)\geq\langle x,y\rangle   \mbox{. }
\nonumber
\end{equation}

For modeling the dry friction phenomenon or the behavior of materials such as clays, 
G\'{e}ry de Saxc\'{e} noticed that it was fruitful to weaken this inequality to 
\begin{equation}
b(x,y)\geq \langle x,y\rangle  \mbox{. }
\nonumber
\end{equation}
Foregoing to separate the function $b(x,y)$ into a sum $\phi(x)+\phi^*(y)$ of two potentials, 
he called it a bipotential. The bipotentials b(x,y) are assumed to be
%
\begin{enumerate}
\item[(i)] convex and lsc in $x$
\bigskip

\item[(ii)] convex and lsc in $y$
\bigskip

\item[(iii)] bounded from below by the duality product: $b(x,y)\geq\langle x,y\rangle$
\end{enumerate}
as it is the case for separable bipotentials. 

%
%
\subsection{Implicit Standard Materials}

A material whose behavior can be described (\cite{bul08,bul110,bul210}) by one of the following three 
implicit constitutive laws
\begin{enumerate}
\item[(iv)] $y$ belongs to the subdifferential of $b(\xi,y)$ with respect to $\xi$ at $x$
\bigskip

\item[(v)] $x$ belongs to the subdifferential of $b(x,\eta)$ with respect to $\eta$ at $y$
\bigskip

\item[(vi)] $b(x,y)=\langle x,y\rangle$
\end{enumerate}
is referred to as an "Implicit Standard Material" (ISM).

%
%
\subsection{Examples of Implicit Standard Materials}
\label{subsec: 4.3}

The Implicit Standard Material model proved to be relevant for describing 
many non--associated phenomena (\cite{sax102} and references contained therein): 
unilateral contact with Coulomb type dry friction (\cite{hji02,hji04,sax91,sax98}),
generalized Drucker--Prager plasticity (\cite{sax102}),
modified Cam--Clay model (\cite{hji08,sax102,zou07,zou10}),
non-associated plasticity of soils (\cite{ber94,ber12}),
non--linear kinematical hardening rule for cyclic plasticity of metals (\cite{arm66,lem90}), 
Lema\^{i}tre's plastic--ductile damage law (\cite{lem87}),
and shakedown analysis of non--standard elastoplastic materials (\cite{bou109,bou209,bou06,sax202}). 

\bigskip
In the next sections (\ref{sec:5}, \ref{sec:6}, \ref{sec:7}, \ref{sec:8}) we recall some 
notions and results on monotone multifunctions (also called multivalued operators or set--valued maps) 
that are well suited to model the behavior of non--associated materials. 

%
%
%
%
%
\section{Basic facts on monotone constitutive laws}
\label{sec:5}

%
%
\subsection{Monotonicity} 
%
%
\begin{definition} 
A constitutive law associated to a multifunction $T$ is monotone 
(\cite{bor00,bor05,mor03,roc270,zal02}) if 
\begin{equation}
[y_1\in Tx_1 \mbox{ and } y_2\in Tx_2 ] \Longrightarrow [\langle x_2-x_1,y_2-y_1\rangle \geq 0 ]  \mbox{. }
\nonumber
\end{equation}
%
%
It is strictly monotone (\cite{zal02}) if, additionally,
\begin{equation}
[y_1\in Tx_1 \mbox{, } y_2\in Tx_2 \mbox{ and } x_2 \not= x_1 ] \Longrightarrow [\langle x_2-x_1,y_2-y_1\rangle > 0 ]  \mbox{. }
\nonumber
\end{equation}
\end{definition}
%
%
\begin{example}
\label{Example 5.2} 
If $A$ is a positive linear mapping from $X$ to $Y$, i.e. 
\begin{equation}
\forall x\in X \mbox{, } \langle x,Ax \rangle \geq 0 \mbox{, } \nonumber
\end{equation}
then the single--valued multifunction $T$ defined by $Tx = \left\{Ax\right\}$ is monotone. 
\end{example}
%
%
\begin{example}
If $T$ is the subdifferential of a convex lsc potential $\phi$, then $T=\partial\phi$ 
is monotone (\cite{bau11,mor03}). However, a monotone multifunction is not necessarily 
the subdifferential of a lsc convex potential. 
\end{example}

%
%
\subsection{Maximality}

%
%
\begin{definition}
A monotone multifunction $T$ is maximal if there is no monotone proper enlargement of $T$ 
(\cite{bau11,bor00,bor05,mor03,zal02}).
\end{definition}
%
%
\begin{remark}
Maximal monotone multifunctions are monotone multifunctions with graphs that cannot 
be enlarged without destroying monotonicity. To establish that a monotone multifunction 
is maximal, one must prove that 
\begin{equation}
(x,y)\not\in G(T) \Longrightarrow [\exists (x_1,y_1)\in G(T) \mbox{, } \langle x-x_1,y-y_1\rangle < 0 ]  \mbox{. }
\nonumber
\end{equation}
\end{remark}
%
%
\begin{remark} 
\label{Remark 5.6}
The maximality assumption is equivalent (\cite{phe98,sim98}) to the statements
\begin{equation}
[(x,y)\in X\times Y \mbox{ and } \forall (x_1,y_1)\in G(T) \mbox{, } \langle x-x_1,y-y_1\rangle \geq 0 ] \Longrightarrow y\in Tx \mbox{, }
\nonumber
\end{equation}
\begin{equation}
[(x,y)\in X\times Y \mbox{ and } \inf_{y_1\in Tx_1}\langle x-x_1,y-y_1\rangle \geq 0 ] \Longrightarrow (x,y)\in G(T) \mbox{. }
\nonumber
\end{equation}
\end{remark}
%
%
\begin{lemma}
\label{basiclemma}
If $T$ is a maximal monotone multifunction, then
\begin{enumerate}
\item[$(i)$] $(x,y)\in G(T) \Longrightarrow \inf_{y_1\in Tx_1}\langle x-x_1,y-y_1\rangle=0$
\bigskip

\item[$(ii)$] $(x,y)\not\in G(T) \Longrightarrow \inf_{y_1\in Tx_1}\langle x-x_1,y-y_1\rangle <0$ 
\bigskip

\item[$(iii)$] $\forall x\in X \mbox{, } \forall y\in Y \mbox{, } \inf_{y_1\in Tx_1}\langle x-x_1,y-y_1\rangle\leq 0$ \mbox{. }
\end{enumerate}
\end{lemma}
\begin{proof}
This lemma is a quantitative version (\cite{phe98,sim98}) of {\bf Remark \ref{Remark 5.6}}.
Assertion $(i)$ follows by taking $(x_1,y_1)=(x,y)$, assertion $(ii)$ is immediate from the 
definition of maximal monotonicity, and assertion $(iii)$ follows from $(i)$ and $(ii)$.
\end{proof}

The following example illustrates the important class of linear maximal mono-- \\
tone multifunctions. 
%
%
\begin{example}
\label{Example 5.8}
Let $X$ and $Y$ be two instances of the same Hilbert space. If $A$ is a positive linear 
(generally not symmetric) mapping from $X$ to $Y$ ({\bf Example \ref{Example 5.2}}), 
then the (single valued) monotone multifunction $T$ defined by $Tx = \left\{Ax\right\}$ is 
automatically maximal monotone (\cite{bau207,phe98,sim98}). 
\end{example}
\begin{proof}
To show this, we must prove that 
\begin{equation}
z=y-Ax \not=0 \Longrightarrow [\exists x_1\in X \mbox{, } \langle x-x_1,y-Ax_1\rangle < 0 ]\mbox{. } 
\nonumber
\end{equation}
Define $x_1=x+\lambda z$ where $\lambda$ is a scalar to be chosen later;
then, by linearity, $y-Ax_1=z-\lambda Az$, and 
\begin{equation}
\langle x-x_1,y-Ax_1\rangle = \lambda ( \lambda \langle z,Az\rangle - \langle z,z\rangle) \mbox{. } 
\nonumber
\end{equation}
If $\langle z,Az\rangle=0$, the expression above is negative for any $\lambda$ positive. 
If $\langle z,Az\rangle$ is positive, the expression above is negative for 
$0<\lambda < \frac{\langle z,z\rangle} {\langle z,Az\rangle}$.
In this second case, we can notice that according to {\bf Lemma \ref{basiclemma}} $(ii)$
\begin{equation}
\inf_{y_1\in Tx_1}\langle x-x_1,y-y_1\rangle \leq \inf_{\lambda} (\lambda^2 \langle z,Az\rangle - \lambda \langle z,z\rangle)= - \frac{1}{4}\frac{\langle z,z\rangle^2} {\langle z,Az\rangle} <0 \mbox{. } 
\nonumber
\end{equation}
We recognize the inverse of Rayleigh's quotient $\frac{\langle z,Sz\rangle} {\langle z,z\rangle}$ 
of the symmetric part $S$ of $A$. Actually, if $S$ is positive definite, then ({\bf Proposition \ref {Proposition 9.2}}) 
%
%
\begin{equation}
\inf_{y_1\in Tx_1}\langle x-x_1,y-y_1\rangle = - \frac{1}{4}\langle z,S^{-1}z\rangle <0 \mbox{. } 
\nonumber
\end{equation}
%
%
\end{proof}

Thus, in the special case of linear mappings, the monotonicity implies the maximal monotonicity. 
Surprisingly, the continuity is also automatically ensured (\cite{sim98}). However, it is worth mentioning 
that the situation is not so simple when $A$ is only defined on a proper linear subspace of $X$ (\cite{phe98,sim98}). 

%
%
\subsection{Fitzpatrick's function}
 
%
%
\begin{definition}
\label{Definition 5.9}
Let $T$ be a maximal monotone multifunction. The associated Fitzpatrick function 
$F_{T, 2}$ is defined (\cite{fit88}) by 
\begin{equation}
F_{T,2}(x,y)=\langle x,y\rangle - \inf_{y_1\in Tx_1}\langle x-x_1,y-y_1\rangle  \mbox{. }
\nonumber
\end{equation}
\end{definition}

%
%
\begin{propos}
\label{Proposition 5.10}
Fitzpatrick's function is bounded from below by the duality product
\begin{equation}
F_{T,2}(x,y)\geq\langle x,y\rangle \nonumber
\end{equation}
and equality is attained if and only if $y\in Tx$.
\end{propos}
\begin{proof}
It is a direct application of {\bf Lemma \ref{basiclemma}}.
\end{proof}

%
%
\begin{theorem}
The Fitzpatrick function $F_{T,2}$ represents the maximal monotone multifunction $T$:
\begin{enumerate}
\item[(i)] $G(T)=\left\{(x,y)\in X\times Y \mid F_{T,2}(x,y)=\langle x,y\rangle\right\}$
\item[(ii)] outside $G(T)$, $F_{T,2}(x,y) > \langle x,y\rangle$ \mbox{. } 
\end{enumerate}
\end{theorem}
\begin{proof}
It is a recast of {\bf Proposition \ref{Proposition 5.10}}.
\end{proof}

%
%
\begin{remark}
\label{Remark 5.12}
$F_{T,2}$ is globally lsc and convex. 
\end{remark}
\begin{proof}
From the definition
\begin{equation}
F_{T,2}(x,y)= \sup_{(x_1,y_1)\in G(T)} [\langle x,y_1\rangle + \langle x_1,y\rangle - \langle x_1,y_1\rangle ]
\nonumber
\end{equation}
is the supremum of a family of continuous affine real--valued functions and therefore is convex 
and lsc on the product space $X\times Y$. 
\end{proof}

%
%
\begin{remark}
With the duality product 
\begin{equation}
\langle \langle (x_1,y_1), (y_2,x_2) \rangle \rangle = \langle x_1,y_2\rangle + \langle x_2,y_1\rangle
\nonumber
\end{equation}
between $X\times Y$ and $Y\times X$, Fitzpatrick's function can be regarded as a Legendre--Moreau--Rockafellar transform. 
\end{remark}
\begin{proof}
From the definition
\begin{equation}
F_{T,2}(x,y)= \sup_{(x_1,y_1)\in G(T)} [\langle \langle (x_1,y_1), (y,x) \rangle \rangle - \langle x_1,y_1\rangle ] 
\nonumber
\end{equation}
$$\quad = \sup_{(x_1,y_1)\in X\times Y} [\langle \langle (x_1,y_1), (y,x) \rangle \rangle - \langle x_1,y_1\rangle - i_{G(T)}(x_1,y_1) ]$$
\\
which is nothing else than the value at $(y,x)$ of the conjugate of the function 
$\langle .,.\rangle + i_{G(T)}$,  where $i_{G(T)}$ is the indicator function of the graph $G(T)$.
\end{proof}

%
%
\begin{example}
Let $X$ and $Y$ be two instances of the same Hilbert space. Let $S$ be a symmetric linear mapping from $X$ to $Y$. 
If $S$ is positive semi--definite, then Fitzpatrick's function of the (single valued) multifunction $T$ 
defined by $Tx = \left\{Sx\right\}$ is 
$$F_{T,2}(x,y) = \begin{cases}
\langle x,y\rangle + \frac{1}{2}\langle y-Sx,\xi\rangle &{\textnormal{ if }} (y-Sx) \in R(S) \\
\quad \quad &{\textnormal{ and }} S\xi = \frac{1}{2}(y-Sx) \\
+ \infty \;\;\;\;\;\;\; &{\textnormal{ if }} (y-Sx) \not \in R(S)
\end{cases}$$
where $R(S)$ is the range of $S$.
\end{example}
\begin{proof}
The supremum with respect to $x_1$ of the concave function $\langle x,y\rangle + \langle x_1-x,y-Sx_1\rangle$ 
is attained when $y-Sx_1 = S(x_1-x)$ or $y-Sx = 2S(x_1-x)$. The result follows with $\xi = x_1-x$ as any solution 
of the linear equation $S\xi = \frac{1}{2}(y-Sx)$ when $y-Sx$ belongs to the range $R(S)$ of $S$. 
\end{proof}

%
%
%
%
%
%
\section{Basic facts on $n$--monotone constitutive laws}
\label{sec:6}

%
%
\subsection{$n$--monotonicity}

%
%
\begin{definition}
\label{definition 6.1}
For $n \geq 2$, a constitutive law associated to a multifunction $T$ is $n$--monotone if (\cite{bar07}) 
\begin{equation} 
\sum_{i=1}^n\langle x_{i+1}-x_{i},y_i\rangle \leq 0   
\nonumber
\end{equation}
for any $n$ elements $(x_i,y_i)$ of $G(T)$ and $(x_{n+1},y_{n+1}) =(x_1,y_1)$.
\end{definition}

%
%
\begin{remark}
We note that $2$--monotonicity reduces to ordinary monotonicity. 
\end{remark}

%
%
\begin{remark}
For $n \geq 2$, $(n+1)$--monotonicity implies $n$--monotonicity. 
\end{remark}
\begin{proof}
Restrain the choice of the $(n+1)$ pairs in the graph $G(T)$ by choosing $(x_{n+1},y_{n+1}) =(x_{n},y_{n})$ 
and $(x_{n+2},y_{n+2})=(x_1,y_1$).
\end{proof}

%
%
\begin{example}
\label{Example 6.4}
\emph{A $2$--monotone example and a $3$--monotone counterexample.}
 
Let $X=Y=\mathbb{R}^2$
and $A$ be the $2\times 2$ matrix 
$\begin{bmatrix}
1 & -2\varepsilon\\
2\varepsilon & 0\\
\end{bmatrix}$ 
where $\varepsilon$ is a nonzero scalar, then the (single valued) multifunction $T$ defined by 
$Tx = \left\{Ax\right\}$ is $2$--mono-
tone but not $3$--monotone.
\end{example}
\begin{proof}
The symmetric part
$S=\begin{bmatrix}
1 & 0\\
0 & 0\\
\end{bmatrix}$ 
of $A$ is positive semi--definite, therefore $T$ is $2$--monotone but not strictly $2$--monotone. 
Take 
$x_1=\begin{bmatrix}
\varepsilon\\
0\\
\end{bmatrix}$,  
$x_2=\begin{bmatrix}
0\\
1\\
\end{bmatrix}$,  
$x_3=\begin{bmatrix}
0\\
0\\
\end{bmatrix}$,   
then
$y_1= A x_1= \varepsilon \begin{bmatrix}
1\\
2 \varepsilon\\
\end{bmatrix}$,  
$y_2 = A x_2= -2 \varepsilon \begin{bmatrix}
1\\
0\\
\end{bmatrix}$,  
$y_3=\begin{bmatrix}
0\\
0\\
\end{bmatrix}$, and   
$\langle x_2-x_1,y_1\rangle + \langle x_3-x_2,y_2\rangle + \langle x_1-x_3,y_3\rangle$
reduces to 
$\langle \begin{bmatrix}
-\varepsilon\\
1\\
\end{bmatrix}, \varepsilon \begin{bmatrix}
1\\
2\varepsilon\\
\end{bmatrix}  \rangle = \varepsilon^2$ which is positive. 

\bigskip
With a positive semi--definite symmetric part and a nonzero skew symmetric part, the linear mapping $A$ cannot 
generate a $3$--monotone multifunction $T$. 
%
\end{proof}

%
%
\subsection{$n$--monotonicity of positive semi--definite symmetric linear mappings}
\label{subsec:6.2}
%
%
\begin{propos}
\label{Proposition 6.5}
Let $X$ and $Y$ be two instances of the same Hilbert space, and $S$ a symmetric linear mapping from $X$ to $Y$. 
If $S$ is positive semi--definite, then the (single valued) multifunction $T$ defined by $Tx = \left\{Sx\right\}$ 
is $n$--monotone for any $n \geq 2$. 
\end{propos} 
\begin{proof}
For $n=3$, in the inequality of the {\bf Definition \ref{definition 6.1}}, let us regard $x_3$ as the origin in $X$ 
and set $x_i = x_3 + z_i$. This translation leads to force
%
%
\begin{equation}
\langle z_1,S z_1\rangle + \langle z_2,S z_2\rangle - \langle z_2,S z_1\rangle
\nonumber
\end{equation}
to be non--negative.
%
Necessarily, $S$ has to be positive semi--definite (take $z_2=0$) and we can extract its positive 
semi--definite square root $S^\frac{1}{2}$. There is no other condition for the $3$--monotonicity because
\begin{equation}
\langle S^\frac{1}{2} z_1,S^\frac{1}{2} z_1\rangle + \langle S^\frac{1}{2} z_2,S^\frac{1}{2} z_2\rangle - \langle S^\frac{1}{2} z_2,S^\frac{1}{2} z_1\rangle 
\nonumber
\end{equation}
is non--negative as well as the $2\times 2$ matrix 
$$\begin{bmatrix}
2 & -1\\
-1 & 2\\
\end{bmatrix} \mbox{.}$$ 
For $n=4$, in the inequality of the {\bf Definition \ref{definition 6.1}}, let us regard $x_4$ as 
the origin in $X$ and set $x_i = x_4 + z_i$. 
This translation leads to force the non--negativity of
\begin{equation}
\langle z_1,S z_1\rangle + \langle z_2,S z_2\rangle + \langle z_3,S z_3\rangle - \langle z_2,S z_1\rangle - \langle z_3,S z_2\rangle \mbox{.} 
\nonumber
\end{equation}
%
%
Necessarily, $S$ has to be positive semi--definite (take $z_2=0$ and $z_3=0$). There is no other condition for 
the $4$--monotonicity because the above expression is non--negative as well as the $3\times 3$ matrix 
$$\begin{bmatrix}
2 & -1 & 0\\
-1& 2  & -1\\
0 & -1 & 2\\
\end{bmatrix} \mbox{.}$$
With the same arguments, one can prove that $T$ is $n$--monotone for every $n \geq 2$. This result is 
consistent with the fact that $T$ is the subdifferential of the convex potential $\phi$ defined by 
$\phi(x)=\frac{1}{2} \langle x,Sx \rangle $  (see {\bf Example \ref{Example 8.3}}). 
%
\end{proof}

%
%
\subsection{$3$--monotonicity of non symmetric linear mappings}
%
%
\begin{propos}
\label{Proposition 6.6}
Let $X=Y$ be a Hilbert space, and $A$ a linear mapping from $X$ to $Y$. $A$ is not symmetric, the 
skew--symmetric part $W$ of $A$ is nonzero, and the symmetric part $S$ of $A$ is assumed to be positive definite. 
The (single valued) multifunction $T$ defined by $Tx = \left\{Ax\right\}$ is strictly $2$--monotone ({\bf Example \ref{Example 5.2}}).
Furthermore, if the linear mapping
\begin{equation}
H_3 = S - \frac{1}{4} A^*S^{-1}A = \frac{3}{4}S - \frac{1}{4}W S^{-1}W^* 
\nonumber
\end{equation}
is positive, then the multifunction $T$ reveals to be $3$--monotone. 
\end{propos}
\begin{proof}
For $n=3$, in the inequality of {\bf Definition \ref{definition 6.1}}, let us regard $x_3$ as the origin in $X$ 
and set $x_i = x_3 + z_i$. 
This translation leads to force the non--negativity of
%
%
\begin{equation}
2\langle z_1,S z_1\rangle + 2\langle z_2,S z_2\rangle - \langle z_2,A z_1\rangle - \langle A^*z_2, z_1\rangle \mbox{.} 
\nonumber
\end{equation}
%
%
Let us set $z_2=\Lambda z_1$, with $\Lambda$ a linear mapping from $X$ to $X$. Then, the non--negativity 
condition is transferred to the symmetric linear mapping 
\begin{equation}
S + \Lambda ^*S\Lambda  - \frac{1}{2}\Lambda ^*A -\frac{1}{2}A^*\Lambda = ( \Lambda ^*S^\frac{1}{2} - \frac{1}{2}A^*S^{-\frac{1}{2}} )( S^\frac{1}{2}\Lambda -     \frac{1}{2}S^{-\frac{1}{2}}A ) + S - \frac{1}{4}A^*S^{-1}A
\nonumber
\end{equation}
for every $\Lambda$. Therefore, the linear mapping $S - \frac{1}{4}AS^{-1}A^*$ has to be positive. It is easy to verify that 
$S - \frac{1}{4}A^*S^{-1}A = \frac{3}{4}S - \frac{1}{4}WS^{-1}W^*$.
\end{proof}
%
%
\subsection{$n$--monotonicity of a non--symmetric $2\times 2$ matrix}
\label{subsec: 6.4}
%
%
\begin{propos}
\label{Proposition 6.7}
Let $A$ be a $2\times 2$ matrix with a positive definite symmetric part $S$ and a nonzero skew--symmetric part 
$$W = \begin{bmatrix}
0 & -r\\
r & 0\\
\end{bmatrix} = 
r\begin{bmatrix}
0 & -1\\
1 & 0\\
\end{bmatrix} = rJ \mbox{.}$$ 
Let us introduce an angle $\theta$ between $0$ and $\frac{\pi}{2}$ such that 
$$|r|= \sqrt{\det S}\, \tan\theta$$
Then, $A$ is $n$--monotone if and only if $$n \theta \leq \pi \mbox{.}$$ 
\end{propos}
\begin{proof}
Let $x_1, x_2, \dots, x_{n}$ be $n$ points in $X$. To study the $n$--monotonicity of $A$, 
we have to verify that the sum 
$$\sum_{i=1}^n\langle x_{i+1}-x_{i},Ax_i\rangle$$
is negative under the closure hypothesis $x_{n+1}=x_1$. Because of the linearity of $A$, this sum 
is invariant by translation. Let us regard $x_n$ as the origin in $X$, then the sum reduces to
$$\sum_{i=1}^{n-1}\langle x_{i+1}-x_{i},Ax_i\rangle = \sum_{i=1}^{n-1}\langle x_{i+1},(S + rJ)x_i\rangle - \sum_{i=1}^{n-1}\langle x_{i},S x_i\rangle \mbox {.}$$
The symmetric part $S$ of $A$ being positive definite, we can substitute each $x_i$ by $S^{-\frac{1}{2}}x_i$, 
and the sum becomes
$$\sum_{i=1}^{n-1}\langle x_{i+1},(I + \frac{r}{\sqrt{\text{det}S}} J)x_i\rangle - \sum_{i=1}^{n-1}\langle x_{i},x_i\rangle $$
%
where we have used the identity $S^{-\frac{1}{2}}JS^{-\frac{1}{2}} = \frac{1}{\sqrt{\text{det}S}} J$. 
Therefore, the $n$--monotonicity of $A$ is governed by the variable $t = \frac{r}{\sqrt{\text{det}S}}$. 
It is equivalent to study the positivity of the real symmetric $(2n-2)\times (2n-2)$ matrix
$$\begin{bmatrix}
   2I  & -I+tJ  & \vdots & 0      & 0\\
-I-tJ  & 2I     & \vdots & 0      & 0\\
\cdots & \cdots & \ddots & \cdots & \cdots \\
     0 & 0      & \vdots & 2I     & -I+tJ\\
     0 & 0      & \vdots & -I-tJ  & 2I\\
\end{bmatrix}$$
or of the Hermitian $(n-1)\times (n-1)$ complex matrix
$$\begin{bmatrix}
2 & -(1-it) & \vdots & 0 & 0\\
-(1+it) & 2  & \vdots & 0 & 0\\
\cdots & \cdots  & \ddots & \cdots & \cdots\\
0& 0  & \vdots & 2 & -(1-it)\\
0& 0  & \vdots & -(1+it) & 2\\
\end{bmatrix} \mbox{.}$$
The eigenvalues of this tridiagonal Hermitian $(n-1)\times (n-1)$ matrix are well--known (\cite{gre78}). They take the $(n-1)$ real values
$$2-2\sqrt{1+t^2} \cos \left(k\frac{\pi}{n}\right) = \frac{2}{\cos\theta}\left[\cos\theta - \cos \left(k\frac{\pi}{n}\right)\right]$$
indexed from $k=1$ to $k=n-1$. The $n$--monotonicity is ensured if the smallest eigenvalue 
$\frac{2}{\cos\theta}\left[\cos\theta - \cos \left(\frac{\pi}{n}\right)\right]$ is positive, and the conclusion follows.
\end{proof}
%
%
\subsection{Maximality}
%
%
\begin{definition}
\label{Definition 6.8}
A $n$--monotone multifunction $T$ is maximal if there is no proper $n$-monotone enlargement of $T$. 
\end{definition}

%
%
\subsection{Fitzpatrick's sequence}

%
%
\begin{definition}
\label{Definition 6.9}
For $n\geq 2$ and $(x,y)\in X\times Y$, let $\left(x_i,y_i\right)$ be 
$n-1$ elements of $G(T)$, indexed from $i=1$ to $n-1$, define $\left(x_n,y_n\right)=(x,y)$, and 
close the loop by $\left(x_{n+1},y_{n+1}\right)=\left(x_1,y_1\right)$. Fitzpatrick's sequence (\cite{bar07}) is defined by
\begin{equation}
F_{T,n}(x,y)=\langle x,y\rangle + 
\sup_{y_i\in Tx_i}\sum_{\lambda=1}^n\langle x_{\lambda +1}-{x_\lambda},y_{\lambda}
\rangle \mbox{.}
\nonumber
\end{equation}
\end{definition} 
%
%
\begin{remark}
For $n=2$, we recover the function $F_{T,2}$, originally proposed by Fitzpatrick to study monotone operators.
\end{remark}

%
%
\begin{propos} 
\label{Proposition 6.11}
Fitzpatrick's sequence is increasing: for $n\geq 3$,
\begin{equation}
F_{T,n}(x,y)\geq F_{T,n-1}(x,y) \mbox{.} 
\nonumber
\end{equation}
\end{propos}
\begin{proof} 
While evaluating the supremum defining $F_{T,n}(x,y)$, choose a sequence 
such that $x_{n-1}=x_{n-2}$ and compare with $F_{T,n-1}(x,y)$.
%
\end{proof}

%
%
\begin{propos}
Every function of Fitzpatrick's sequence is bounded from below by the duality product:
\begin{equation}
F_{T,n}(x,y)\geq\langle x,y\rangle 
\nonumber
\end{equation}
and equality is attained if and only if $y\in Tx$.
\end{propos}
\begin{proof}
This is a generalization of {\bf Proposition \ref{Proposition 5.10}} by {\bf Proposition \ref{Proposition 6.11}}.
%
\end{proof}

%
%
\begin{theorem} 
Let $T:X\longrightarrow 2^Y$ be maximal $n$--monotone, for some $n\geq 2$; then, $F_{T,n}$ represents $T$:

\begin{enumerate}
\item[(i)] $G(T)=\left\{(x,y)\in X\times Y \mid F_{T,n}(x,y)=\langle x,y\rangle\right\}$
\item[(ii)] outside $G(T)$, $F_{T,n}(x,y) > \langle x,y\rangle$ \mbox{. } 
\end{enumerate}
\end{theorem}
\begin{proof}
We refer the reader to (\cite{bar07,bau107}), where it is shown that $F_{T,n}$ 
is well suited to study $n$--monotone operators.
\end{proof}

%
%
\begin{remark}
$F_{T,n}$ is globally lsc and convex. 
\end{remark}
\begin{proof}
From the definition, as $F_{T,2}$ was globally lsc and convex ({\bf Remark \ref{Remark 5.12}}), $F_{T,n}$ is the upper hull 
of a family of continuous and affine real--valued functions and therefore is convex and lsc on the 
product space $X\times Y $.
\end{proof}

We present next a recursive formula allowing in certain cases an iterative computation of 
Fitzpatrick's sequence of a monotone operator. 

%
%
%
%
%
%
\section{Recursion formula}
\label{sec:7} 
%
%
\begin{propos}
\label{Proposition 7.1} 
Let $T:X\longrightarrow 2^Y$ be a $(n+1)$--monotone multifunction ($n\geq 2$), then
\begin{equation}
\forall (x,y)\in X\times Y \mbox{, }  \qquad F_{T,{n+1}}(x,y)= \sup_{\eta \in T \xi} [F_{T,n}(\xi,y) + \langle x - \xi,\eta \rangle ] \mbox{. } 
\nonumber
\end{equation}
\end{propos}
\begin{proof}
For $n\geq 2$ and $(\xi,\eta)\in X\times Y$, let $(x_i,y_i)$ be $n-1$ elements of $G(T)$ 
indexed from $i=1$ to $n-1$.
By definition,
%
%
\bigskip

$F_{T,n}(\xi,y) + \langle x - \xi,\eta \rangle =$ 
$$ = \sup_{y_i\in Tx_i} (\sum_{i=1}^{n-2} \langle x_{i+1}-x_i,y_i \rangle + \langle x_1,y \rangle + \langle \xi - x_{n-1},y_{n-1} \rangle ) 
+ \langle x - \xi,\eta \rangle \mbox{. }$$
Following the supremum over the $n-1$ elements $(x_i,y_i)$ of $G(T)$, by supremizing over 
$(\xi,\eta)\in G(T)$, we recognize (\cite{bau307}) 
$$F_{T,{n+1}}(x,y) = \displaystyle{ \sup_{\substack{y_i\in Tx_i \\ y_n\in Tx_n}} } ( \displaystyle{ \sum_{i=1}^{n-2} } \langle x_{i+1}-x_i,y_i \rangle 
+ \langle x_n - x_{n-1},y_{n-1} \rangle + \langle x_1,y \rangle + \langle x - x_n,y_n \rangle )$$
where the role of $(x_n,y_n)$ is played by $(\xi,\eta)$. 
\end{proof}

%
%
%
%
%
%
\section{Fitzpatrick's sequence for a GSM}
\label{sec:8}

%
%
\begin{definition} 
A constitutive law associated to a multifunction $T$ is cyclically monotone if it is 
$n$--monotone for every $n\geq 2$. 
\end{definition}

%
%
\begin{example}
If $X=Y$ is a Hilbert space, and $A$ is a symmetric positive semi--definite linear mapping 
from $X$ to $Y$, then the (single valued) multifunction $T$ defined by $Tx = \left\{Ax\right\}$ 
is cyclically monotone (see {\bf Proposition \ref{Proposition 6.5}}).  
\end{example}

%
%
\begin{example}
\label{Example 8.3} 
If $T$ is the subdifferential of a convex lsc potential $\phi$, then $T=\partial\phi$ 
is cyclically monotone (\cite{roc66}).
\end{example}
%
%
%
\begin{definition} 
A cyclically monotone multifunction $T$ is maximal if $T$ is cyclically monotone and 
no proper extension of T is cyclically monotone.
\end{definition}

%
%
\subsection{Recovery of the bipotential for a GSM}

A convex potential which has at least one finite value (not identically $+\infty$) is called proper. 
When the constitutive law of a GSM is described by a convex, lsc, and proper potential $\phi$, the 
multifunction $T=\partial\phi$ is maximal monotone and cyclically monotone, hence maximal cyclically monotone.
For $n\geq 2$, every $n$--monotonicity is captured (\cite{bar07}) by Fitzpatrick's function $F_{T,n}(x,y)$. 
Moreover, Fitzpatrick's sequence admits a pointwise limit
\begin{equation}
F_{T,\infty}(x,y)=\sup_{n\geq 2}F_{T,n}(x,y) 
\nonumber
\end{equation}
which is nothing else than the sum of the potential and its conjugate (\cite{bar07,bau11,bor10})
\begin{equation}
F_{T,\infty}(x,y)=\phi(x) + \phi^*(y)
\nonumber
\end{equation}
and we recover the GSM-specific separated bipotentials.

Conversely, when a multifunction $T$ is maximal cyclically monotone, a constructive theorem (\cite{bau11}) 
due to R.T. Rockafellar (\cite{roc170,roc66}) and J.J. Moreau (\cite{mor03}) proves the existence 
of a convex, lsc, and proper potential $\phi$ such that $T=\partial\phi$. The method for retrieving $\phi(x)$ 
consists in fixing $y$ in $F_{T,\infty}(x,y)$. By duality, the method for retrieving $\phi^*(y)$ consists in fixing 
$x$ in $F_{T,\infty}(x,y)$. Actually, the sequence constructed by Fitzpatrick is a genial rewriting of the 
Moreau--Rockafellar theorem. 

%
%
\subsection{Fitzpatrick's sequence for an indicator function}
When the potential $\phi$ is the indicator function $i_K$ of a convex $K$, then Fitzpatrick's sequence of 
the multifunction $T=\partial\phi$ is reduced to a single element (\cite{bar07})
\begin{equation}
\forall n\geq 2, F_{T,n}(x,y)=F_{T,\infty}(x,y)=i_K(x)+i_K^*(y)  \mbox{. }
\nonumber
\end{equation}
%
%

%
%
\subsection{Fitzpatrick's sequence for a support function}

By duality, the same equality holds (\cite{bar07}) when the potential $\phi$ is the 
support function $i_K^*$ of a convex set $K$. 

\bigskip
In the next section we will emphasize (and sometimes repeat) in the framework of continuum mechanics 
some notions and results concerning the linear monotone operators previously presented 
in sections \ref{sec:5}, \ref{sec:6} and \ref{sec:8}.

%
%
%
%
%
%
\section{Linear constitutive laws}
\label{sec:9}

%
%
\subsection{Symmetric Linear Constitutive Laws}

As in {\bf Example \ref{Example 5.8}}
let $X$ and $Y$ be two instances of the same Hilbert space. A linear 
symmetric law $y = Sx$ models a Standard Material if and only 
if $S$ is symmetric (i.e. coincides with its adjoint $S^*$). The potential is then 
$\phi(x)=\frac{1}{2}\langle x,Sx\rangle$. 
When $S$ is also positive definite, this potential is strictly convex and the symmetric linear 
constitutive law $y = Sx$ models a well--known class of GSM: the linear elastic materials. 

%
%
\begin{propos}
\label{Proposition 9.1}
Fitzpatrick's sequence of the linear symmetric positive definite law $y = Sx$ is
\begin{equation}
F_{S,n}(x,y)=\langle x,y\rangle + \frac{n-1}{2n} \| S^{-\frac{1}{2}}y - S^{\frac{1}{2}}x \| ^2 
\nonumber
\end{equation}
$$ = \frac{1}{n}\langle x,y\rangle + (1 - \frac{1}{n} ) ( \frac{1}{2} \langle x,Sx\rangle + \frac{1}{2} \langle y,S^{-1}y \rangle )$$
$$= \langle x,y\rangle + (1 - \frac{1}{n} ) \frac{1}{2} \langle y-Sx,S^{-1}(y-Sx) \rangle  \quad \forall n\geq 2$$
where $S^{\frac{1}{2}}$ and $S^{-\frac{1}{2}}$ are the square roots of $S$ and its inverse 
$S^{-1}$. The pointwise limit of Fitzpatrick's sequence is
\begin{equation}
F_{S,\infty}(x,y)=\frac{1}{2}\langle x,Sx\rangle + \frac{1}{2} \left\langle y,S^{-1}y \right\rangle   \mbox{. }
\nonumber
\end{equation}
\end{propos}
\begin{proof}
If in {\bf Definition \ref{Definition 6.9}} of $F_{S,n}(x,y)$ we replace each pair 
$\left(x_{\lambda},y_{\lambda}\right)$
by the pair $\left(S^{\frac{1}{2}}x_{\lambda},S^{-\frac{1}{2}}y_{\lambda}\right)$, then
we observe that it is enough to give the proof for $S$ being the linear identity mapping $I$. 
But a proof for $S=I$ was given in \cite{bar07}. See also {\bf Propositions \ref{Proposition 9.3}}, {\bf \ref{Proposition 9.4}} and {\bf Example \ref{Example 9.5}} in this Section \ref{sec:9}.
%
\end{proof}

%
%
\subsection{Linear Elasticity}

The phenomenon described above is that of elasticity: $x$ is the strain tensor, $y$ is the stress tensor, 
$S$ is the stiffness tensor, its inverse $S^{-1}$ is the compliance tensor. The behavior of the elastic 
material is governed by the potential $\phi(x) = \frac{1}{2}\langle x,Sx\rangle$ or by the conjugate potential 
$\phi^*(y) = \frac{1}{2} \left\langle y,S^{-1}y \right\rangle $.

Most numerical solutions of the linear elasticity problems are obtained (\cite{cia11,fra65,val81,val77,was82}) 
by applying variational methods: 
\begin{enumerate}
\item[--] the minimum principle of the potential energy, involving an integral functional of the kinematically 
admissible displacements fields with the integrand $\phi(x)$ regarded as strain energy density,
\item[--] the minimum dual principle of the complementary energy, involving an integral functional of the statically 
admissible stress fields with the integrand $\phi^*(y)$ regarded as stress energy density,
\item[--] a primal--dual two--field variational principle, involving an integral functional with the integrand $\phi(x)+\phi^*(y)$, 
\item[--] three--field variational principles, etc. 
\end{enumerate}

The bipotential approach, very close to the primal--dual two--field methods, suggests intermediate variational 
principles involving integral bifunctionals (\cite{bul12,mat11,sax102,zou07}) with an integrand 
\begin{equation}
F_{S,n}(x,y)=\langle x,y\rangle + (1 - \frac{1}{n} ) [ \phi(x)+\phi^*(y)-\langle x,y\rangle ] \mbox{.}
\nonumber
\end{equation}
%
%

%
%
\subsection{Strictly Monotone Non--Symmetric Linear Constitutive Laws}

A linear law $y = Ax$ can be monotone without being symmetric, it is only necessary 
that the symmetric part $S=\frac{1}{2}\left(A+A^*\right)$ of $A$ is positive. It will be 
strictly monotone if $S$ is positive definite. 

%
%
\begin{propos}
\label{Proposition 9.2}
Fitzpatrick's function of a strictly monotone linear constitutive law $y = Ax$ is
\begin{equation}
F_{A,2}(x,y) = \langle x,y\rangle + \frac{1}{4}\langle y- Ax,S^{-1}(y- Ax)\rangle \mbox{.}
\nonumber
\end{equation}
\end{propos}
\begin{proof}
If $y_1 = Ax_1$, the infimum of 
\begin{equation}
\langle x-x_1,y-y_1\rangle = \langle x-x_1,y- Ax\rangle + \langle x-x_1,S(x-x_1)\rangle 
\nonumber
\end{equation}
is attained for $x_1$ solution of $2S(x-x_1) = Ax-y$. Therefore, the first element of  
Fitzpatrick's sequence is
\begin{equation}
F_{A,2}(x,y) = \langle x,y\rangle + \frac{1}{4}\langle y- Ax,S^{-1}(y- Ax)\rangle \mbox{.}
\nonumber
\end{equation}
This function is a bipotential representing the non--associated linear law $y = Ax$.
\end{proof}

%
%
\begin{propos}
\label{Proposition 9.3}
Fitzpatrick's function of a (not strictly) monotone linear constitutive law $y = Ax$ is
\begin{equation}
F_{A,2}(x,y) = \langle x,y\rangle + \frac{1}{2}\langle \xi,(y- Ax)\rangle + i_{R(S)}(y- Ax)
\nonumber
\end{equation}
with $\xi$ any solution of the linear equation $2S \xi = y- Ax$. 
\end{propos}
\begin{proof}
If the symmetric part $S$ of $A$ is positive, but not positive definite, the linear equation 
\begin{equation}
2S(x_1-x) = y - Ax 
\nonumber
\end{equation}
can have solutions $x_1-x=\xi$ if and only if $y - Ax$ belongs to the range $R(S)$ of $S$. Therefore, 
the first element of Fitzpatrick's sequence is
\begin{equation}
F_{A,2}(x,y) = \langle x,y\rangle + \frac{1}{2}\langle \xi,(y- Ax)\rangle + i_{R(S)}(y- Ax)
\nonumber
\end{equation}
and this value does not depend on the choice of the solution $\xi$ of the linear equation $2S \xi = y- Ax$. 
\end{proof} 

%
%
%
\subsection{Strictly n--Monotone Linear Constitutive Laws} 

%
%
\begin{propos}
\label{Proposition 9.4}
Let $n\geq 2$, and $A$ be a strictly $n$--monotone linear mapping, then 
\begin{equation}
F_{A,n}(x,y) = \langle x,y\rangle + \frac{1}{4}\langle y- Ax,H^{-1}_n(y- Ax)\rangle 
\nonumber
\end{equation}
where the linear symmetric mappings $H_k$ indexed from $k=2$ to $n$ are generated by 
the recursive formula $H_{k+1} = S - \frac{1}{4} A^* H^{-1}_k A$, with an initial value $H_2 = S$.
\end{propos}
\begin{proof}
The function $\Phi(x,y) = F_{A,n}(x,y) - \langle x,y\rangle$ is the supremum of
\begin{equation}
\sum_{i=1}^{n-2} \langle x_{i+1}-x_i,Ax_i \rangle + \langle x_1-x,y\rangle + \langle x-x_{n-1},Ax_{n-1}\rangle  \mbox{. }
\nonumber
\end{equation}
Let us regard $x$ as the origin in $X$ and set $x_i = x + z_i$, this translation leads to supremize
%
%
\begin{equation}
\sum_{i=1}^{n-2} \langle z_{i+1}-z_i,Ax + Az_i \rangle + \langle z_1,y\rangle - \langle z_{n-1},Ax + Az_{n-1}\rangle \mbox{. }  
\nonumber
\end{equation}
The sum $\displaystyle{ \sum_{i=1}^{n-2} } \langle z_{i+1}-z_i,Ax \rangle$ reduces to $\langle z_{n-1}-z_1,Ax \rangle$, therefore
\begin{equation}
\Phi(x,y) = \sup_{z_i} [ \sum_{i=1}^{n-2} \langle z_{i+1}-z_i,Az_i \rangle + \langle z_1,y-Ax\rangle - \langle z_{n-1},Az_{n-1}\rangle ] 
\nonumber
\end{equation}
$$= \sup_{z_i} [ \sum_{i=1}^{n-2} \langle z_{i+1},Az_i \rangle - \sum_{i=1}^{n-1} \langle z_i,Sz_i\rangle + \langle z_1,y-Ax\rangle ] \mbox{. }$$ 
Due to the strict monotonicity of $A$, the bracketed term is a strictly concave and differentiable 
function in $(z_1,z_2,\ldots,z_{n-1})$. Then, by differentiation, the maximum is attained 
for $(z_1,z_2,\ldots,z_{n-1})$ solution of the linear system
 \begin{equation*}\label{Multi2}
 \begin{split}
 -2Sz_1+A^*z_2&=Ax-y\\
  Az_{i-1}-2Sz_i + A^*z_{i+1} &=0 \quad \quad \quad \quad \quad \quad \quad \quad \text{for}\; i= 2 \;\text{to}\; n-2  \\
 Az_{n-2}-2Sz_{n-1}&=0.
   \end{split} 
\end{equation*}
When these stationarity conditions are satisfied, the quadratic part of the above concave function 
is the additive inverse of the half of the linear part, and \cite{val12}
\begin{equation}
F_{A,n}(x,y) = \langle x,y\rangle + \frac{1}{2}\langle z_1,y- Ax\rangle \mbox{. }
\nonumber
\end{equation}
It remains to express $z_1$ in terms of $y- Ax$. Let us rewrite the above linear system as
 \begin{equation*}\label{Multi2b}
 \begin{split}
 Sz_1-\frac{1}{2}A^*z_2&=\frac{1}{2}\left(y-Ax\right)\\
 Sz_i-\frac{1}{2}A^*z_{i+1}&=\frac{1}{2}Az_{i-1} \quad \quad \quad \quad \quad \quad \quad \quad \text{for}\; i= 2 \;\text{to}\; n-2\\
 Sz_{n-1}&=\frac{1}{2}Az_{n-2}. 
   \end{split}
\end{equation*} 
Clearly, this linear system can be solved by backward substitution. Let us intitialize by 
$H_2=S$ a sequence $H_k$ (indexed from $k=2$ to $n$) of invertible linear symmetric mappings, and set
\begin{equation}
z_i = \frac{1}{2}H^{-1}_{n+1-i}Az_{i-1}
\nonumber
\end{equation}
then successively: 
\begin{equation}
H_{n-i+1} = S -\frac{1}{4}A^*H_{n-i}^{-1}A  \quad \quad \quad \quad \quad \quad \quad \quad  \text{for}\; i= n-2 \;\text{to}\; 1
\nonumber
\end{equation}
We end with
\begin{equation}
H_n z_1 = \frac{1}{2}(y-Ax)
\nonumber
\end{equation}
and the conclusion follows. 
\end{proof} 

%
%
\begin{example}
\label {Example 9.5}
If $A$ is the identity $I$, the sequence $H_k$ is generated by the recursive formula 
$H_{k+1} = I - \frac{1}{4}H_k^{-1}$. Because of the initial value $H_2=I$, each $H_k$ is spheric, 
and we can set $H_k=\alpha_k I$. The sequence of scalars $\alpha_k$ is generated by the recursive 
formula $\alpha_{k+1} = 1-\frac{1}{4\alpha_k}$. The homographic function $1-\frac{1}{4\alpha}$ admits 
the unique fixed point $\alpha=\frac{1}{2}$. As a property of the complex projective line 
$\mathbb {P}^1(\mathbb {C})$, 
the sequence $\beta_k = \frac{1}{\alpha_k - \frac{1}{2}}$ is an arithmetic progression. 
The common difference is equal to $2$: $\beta_{k+1} = \beta_k+2$. The initial term being $\alpha_2=1$, 
i.e. $\beta_2=2$, the solution is $\beta_k = 2(k-1)$ i.e. $\alpha_k = \frac{1}{2} \frac{k}{k-1}$. Thus 
$H_n^{-1} = 2 \frac{n-1}{n} I$ and we retrieve the result given in \cite{bar07} 
\begin{equation}
F_{I,n}(x,y) = \langle x,y\rangle + ( 1-\frac{1}{n} ) \frac{1}{2} \langle y-x,y-x\rangle  \mbox{. }
\nonumber
\end{equation}
\end{example}

%
%
\begin{example}
\label{Example 9.6}
If $A=S$, the sequence $H_k$ is generated by the recursive formula $H_{k+1} = S - \frac{1}{4}SH_k^{-1}S$. 
Because of the initialization $H_2=S$, each $H_k$ is proportional to $S$, and we can set $H_k=\alpha_k S$. 
The sequence of scalars $\alpha_k$ is as in {\bf Example \ref {Example 9.5}}.
Thus $H_n^{-1} = 2 \frac{n-1}{n} S^{-1}$ and we retrieve the result of {\bf Proposition \ref{Proposition 9.1}}   
\begin{equation}
F_{S,n}(x,y) = \langle x,y\rangle + ( 1-\frac{1}{n} ) \frac{1}{2} \langle y-Sx,S^{-1}(y-Sx)\rangle \mbox{. } 
\nonumber
\end{equation}
\end{example}

%
%
\begin{remark}
The symmetric linear mappings $H_k$ are positive definite. 
\end{remark}

%
%
\begin{remark}
The recursive construction $H_{k+1} = S - \frac{1}{4}A^* H_k^{-1} A$ stops if $H_k$ 
ceases to be positive definite, i.e. if $A$ fails to be $(k+1)$--monotone. 
\end{remark}

%
%
\begin{remark}
The function $F_{A,3}$ captures the $3$--monotonicity ({\bf Example \ref{Example 6.4}}) of the linear mapping $A$.   
\end{remark}

%
%
\begin{remark}
The special recursion formula $H_{n+1} = S - \frac{1}{4}A^* H_n^{-1} A$ must be in accordance with 
the general recursion formula ({\bf Proposition \ref{Proposition 7.1}})
%
%
\begin{equation}
F_{A,n+1}(x,y)= \sup_{\xi} \left[ F_{A,n}(\xi,y) + \langle x-\xi,A\xi \rangle \right]  \mbox{. }
\nonumber
\end{equation}
To evaluate this supremum, let us regard $x$ as the origin in $X$ and set $\xi=x+\zeta$. 
This translation leads to supremize the function
%
\begin{equation}
\langle x+\zeta,y \rangle -  \langle \zeta,Ax+A\zeta \rangle +\frac{1}{4} \langle y-Ax-A\zeta,H_n^{-1}(y-Ax-A\zeta) \rangle 
\nonumber
\end{equation}
with respect to $\zeta$. The quadratic part of this function reads
\begin{equation}
- \langle \zeta,S\zeta \rangle +\frac{1}{4} \langle A\zeta,H_n^{-1}A\zeta \rangle = - \langle \zeta,H_{n+1}\zeta \rangle  \mbox{. }
\nonumber
\end{equation}
The linear part of this function reads $\langle \zeta,z \rangle$ with 
\begin{equation}
z = y-Ax - \frac{1}{2}A^* H_n^{-1}(y-Ax) = ( I- \frac{1}{2}A^* H_n^{-1} )(y-Ax)
\nonumber
\end{equation}
%
%
The constant part of this function reads
\begin{equation}
\langle x,y \rangle +\frac{1}{4} \langle y-Ax,H_n^{-1}(y-Ax) \rangle \mbox{. }
\nonumber
\end{equation}
If $A$ is a strictly $(n+1)$--monotone linear mapping, the function to maximize is strictly concave. 
The maximum is attained for $\zeta$ solution of the linear system
\begin{equation}
2H_{n+1} \zeta = z  \mbox{. }
\nonumber
\end{equation}
When this stationarity condition is satisfied, the quadratic part is the additive inverse of the half of the linear part, and 
\begin{equation}
F_{A,n+1}(x,y)= \langle x,y \rangle +\frac{1}{4} \langle y-Ax,H_n^{-1}(y-Ax) \rangle +\frac{1}{4} \langle z,H_{n+1}^{-1}z \rangle 
\nonumber
\end{equation}
has to coincide with 
\begin{equation}
F_{A,n+1}(x,y)= \langle x,y \rangle +\frac{1}{4} \langle y-Ax,H_{n+1}^{-1}(y-Ax) \rangle \mbox{. }
\nonumber
\end{equation}
Therefore, the sequence $H_k^{-1}$ follows the formula
\begin{equation}
H_n^{-1} + ( I- \frac{1}{2} H_n^{-1} A ) H_{n+1}^{-1} ( I- \frac{1}{2}A^* H_n^{-1} ) = H_{n+1}^{-1}
\nonumber
\end{equation}
which allows us to compute recursively $K_{n+1} = H_{n+1}^{-1}$  from $K_{n} = H_{n}^{-1}$ 
(hence from $K_2=S^{-1}$) by solving the equation
\begin{equation} 
K = K_n + ( I- \frac{1}{2} K_n A ) K ( I- \frac{1}{2}A^* K_n )
\nonumber
\end{equation}
using a fixed point algorithm initialized by $K_n$.
\end{remark}
%
%
%
%
%

\section{Study of linear monotone coaxial constitutive laws}
\label{sec:10}

%
%
\subsection{Coaxial constitutive laws}

%
%
\begin{definition}
Let $X$ and $Y$ be the $6$--dimensional Euclidean space of real symmetric $3\times 3$ matrices 
(with $e$ as identity matrix and $\langle x,y\rangle =\mbox{tr}(xy)$ as duality product). 
The variables $x$ and $y$ can be regarded as strain and stress tensors. A constitutive law 
relating $x$ and $y$ is coaxial if $x$ and $y$ have the same eigenvectors. 
\end{definition}

%
%
\begin{propos}
\label{Proposition 10.2} 
Under the additional assumption of linearity, the coaxial constitutive laws have the general form
\begin{equation}
y=\left[\mbox{tr}(kx)\right]e+2\mu x \nonumber
\end{equation}
where $\mu$ is a scalar and $k$ is a symmetric $3\times 3$ matrix.
\end{propos}
\begin{proof}
Let us begin by making two remarks
\begin{enumerate}
\item[(i)] if $x = e$, then $y$ is spheric, 
\bigskip

\item[(ii)] for any unit vector $u$, if $x = uu^*$, then $y = \alpha(u)uu^* +\beta(u)e$ where 
$\alpha$ and $\beta$ are two scalar functions. 
\end{enumerate}
Consider 3 unit vectors $u$, $v$, $w$  constituting an orthonormal basis. Taking 
\begin{equation}
x = uu^* + vv^* +ww^* = e \mbox{, }
\nonumber
\end{equation}
we can conclude from the linearity assumption that
\begin{equation}
\alpha(u)uu^* + \alpha(v)vv^* + \alpha(w)ww^* + [\beta(u) + \beta(v) + \beta(w)]e
\nonumber
\end{equation}
is spheric. This implies that the scalar function $\alpha$ has a constant value that 
we will denoted by $2\mu$. Applying now the linearity assumption to the generic case 
\begin{equation}
x = \lambda_1 uu^* + \lambda_2 vv^* + \lambda_3 ww^* \mbox{, }
\nonumber
\end{equation}
we conclude that $y = 2\mu x + \gamma(x)e$ where $\gamma$ is a scalar linear function. 
But any linear scalar function of $x$ can be expressed as $\mbox{tr}(kx)$ where $k$ is a symmetric 
$3\times 3$ matrix, and the conclusion follows. 
\end{proof}

%
%
\begin{remark}
A linear coaxial constitutive law involves $7$ coefficients: the scalar $\mu$ and 
$6$ independent coefficients of the symmetric matrix $k$. 
\end{remark}

%
%
\begin{remark} 
\label{Remark 10.4}
A linear coaxial constitutive law is not symmetric except when the matrix $k$ is spheric 
($k = \lambda e$), in which case $\lambda$ and $\mu$ are Lam\'{e}'s coefficients of Hooke's 
elastic constitutive law
\begin{equation}
y=\lambda(\mbox{tr}\,x)e+2\mu x  \mbox{. }
\nonumber
\end{equation}

It is a well--known result that Hooke's constitutive law is positive if and only if
\begin{equation}
3\lambda +2\mu\geq 0 \mbox{ and } \mu\geq 0  \mbox{. }
\nonumber
\end{equation}
\end{remark}

%
%
\subsection{Symmetric part of a linear coaxial constitutive law} 

If the deviatoric part $h$ of the matrix $k$ is not $0$, set
\begin{equation}
k=\lambda e+h \mbox{ with } \lambda=\frac{1}{3}(\mbox{tr}\, k) \mbox{ and } \mbox{tr}\, h=0 \mbox{. }
\nonumber
\end{equation}

The symmetric part of the linear coaxial constitutive law is then
\begin{eqnarray}
Sx &=& \frac{1}{2}\left[\mbox{tr}(kx)\right]e+\frac{1}{2}(\mbox{tr}\,x)k+2\mu x \nonumber \\
   &=& \lambda(\mbox{tr}\,x)e +
        2\mu x+\frac{1}{2}(\mbox{tr}\,x)h+\frac{1}{2}\left[\mbox{tr}(hx)\right]e \nonumber \mbox{. }
\end{eqnarray}

%
%
\subsection{Monotonicity of linear coaxial constitutive laws}

%
%
\begin{propos}
\label{Proposition 10.5}
A linear coaxial constitutive laws is monotone if and only if \textnormal{(\cite{val09,val12})}
\begin{equation}
3\lambda +2\mu\geq 0 \mbox{, } \mu\geq 0 \mbox{ and } \mbox{tr}\left(h^2\right)\leq \displaystyle\frac{8}{3}\mu(3\lambda +2\mu) \mbox{. }
\nonumber
\end{equation}
\end{propos}

%
%
\begin{remark}
 In addition to the usual conditions on Lam\'{e}'s coefficients, the mono-
tonicity condition of the coaxial constitutive law demands that the deviatoric part $h$ 
of the matrix $k$ has not to be too large.
\end{remark} 
\begin{proof}
If the deviatoric tensor $h$ is not vanishing, let us express the matrix of the symmetric part $S$ 
of the linear coaxial constitutive law in an orthonormal basis constituted of four deviatoric tensors 
$d_1$, $d_2$, $d_3$, $d_4$ orthogonal to $h$, $d_5=\frac{h}{\| h\|}$ and $d_6 = \frac{e}{\sqrt{3}}$. 
For $i=1$ to $4$, $Sd_i = 2 \mu d_i$, otherwise $Sd_5 = 2 \mu d_5 + \frac{1}{2} \| h\| e$, 
and $Se = (3\lambda +2\mu)e +\frac{3}{2}h$. The matrix of $S$ is therefore the symmetric block matrix
$$\begin{bmatrix}
2 \mu I_4 & 0 & 0\\
0& 2 \mu & \frac{\sqrt{3}}{2} \| h\|\\
0 & \frac{\sqrt{3}}{2} \| h\| & 3\lambda +2\mu\\
\end{bmatrix}$$
where $I_4$ is the $4\times 4$ identity matrix. The $6\times 6$ symmetrical matrix representing $S$ 
is positive if and only if the scalar $\mu$ and the $2\times 2$ symmetrical matrix 
$$s = \begin{bmatrix}
2 \mu & \frac{\sqrt{3}}{2} \| h\|\\
\frac{\sqrt{3}}{2} \| h\| & 3\lambda +2\mu\\
\end{bmatrix}$$
are positive. The additional conditions for $s$ to be positive are 
\begin{equation}
3\lambda +2\mu \geq 0 \mbox{ and } \frac{3}{4} \| h\|^2 \leq 2 \mu (3\lambda +2\mu)  \mbox{. }
\nonumber
\end{equation}
%
%
\end{proof}

%
%
\subsection{A characteristic angle of monotone linear coaxial laws}
\label{subsec: 10.4}
Let us rewrite the inequality limiting the magnitude of $h$ as follows:
\begin{equation}
\| h\| \leq \frac{2}{\sqrt{3}}  \sqrt {2 \mu (3\lambda +2\mu)}  \mbox{. }
\nonumber
\end{equation}
We are led to introduce an angle $\theta$  between $0$ and $\frac{\pi}{2}$ such that
\begin{equation}
\| h\| = \frac{2}{\sqrt{3}}  \sqrt {2 \mu (3\lambda +2\mu)}\sin \theta  \mbox{. }
\nonumber
\end{equation}
The smaller $\theta$ is, the closer is the linear coaxial law to Hooke's constitutive law. 
If $2 \theta < \pi$, then the law is strictly monotone. 

%
%
\subsection{Condition of $n$--monotonicity for a linear coaxial law}
\label{subsec: 10.5}

%
%
\begin{propos}
\label{Proposition 10.7}
A linear coaxial law is $n$--monotone if and only if
$$3\lambda +2\mu \geq 0 \mbox{, } \mu \geq 0  \mbox{ and } n \theta \leq \pi \mbox{.}$$
\end{propos}
\begin{proof}
A monotone linear coaxial law is represented by a $6\times 6$ block matrix, the diagonal of which is 
compound of a positive definite spheric $4\times 4$ matrix $2\mu I_4$ and a $2\times 2$ matrix 
$$a = \begin{bmatrix}
2 \mu & 0\\
\sqrt{3} \| h\| & 3\lambda +2\mu\\
\end{bmatrix} =
s + \frac{\sqrt{3}}{2}\| h\| \begin{bmatrix}
0 & -1\\
1 & 0\\
\end{bmatrix} = s +rJ$$
with
$$r = \frac{\sqrt{3}}{2}\| h\| = \sqrt{\text{det}\,s} \tan\theta \mbox{. }$$
According to {\bf Proposition \ref{Proposition 6.7}}, the additional condition for the constitutive 
law to be $n$--monotone is $n \theta \leq \pi $. 
\end{proof}

%
%
\subsection{Condition of cyclic--monotonicity for a linear coaxial law}

%
%
\begin{propos}
\label{Proposition 10.8}
A linear coaxial law which is cyclically monotone reduces to Hooke's constitutive law. 
\end{propos}
\begin{proof}
If a linear coaxial law is cyclically monotone, then $n \theta \leq \pi$ for every integer $n$ 
larger than $2$. The field of the real numbers being Archimedean, this is possible only for $\theta =0$, 
i.e. $h=0$. Therefore the linear coaxial law reduces to Hooke's constitutive law. 
\end{proof}

%
%
\subsection{Fitzpatrick's sequence of a strictly $n$--monotone linear coaxial law}

%
%
\begin{propos} 
Let $A$ be a strictly $n$--monotone linear coaxial law, then Fitzpatrick's sequence of $A$ is finite 
and constituted of the $n-1$ functions $F_{A,k}$ defined for $k = 2$ to $n$ by
\begin{equation}
F_{A,k}(x,y) = \mbox{tr}(xy) + \frac{1}{4} \mbox{tr} \left[ (y-Ax) H^{-1}_k (y-Ax) \right]
\nonumber
\end{equation}
with
$$H_k = \begin{bmatrix}
2 \mu \alpha_k I_4 & 0\\
0 & h_k\\
\end{bmatrix}
= \begin{bmatrix}
\frac{k}{k-1} \mu I_4 & 0\\
0 & \frac{1}{2} \frac{\sin(k\theta)}{\sin((k-1)\theta)} \frac{1}{\cos \theta} s\\
\end{bmatrix}
$$
\end{propos}
\begin{proof}
The orthonormal basis $d_1$, $d_2$, $d_3$, $d_4$, $d_5=\frac{h}{\| h\|}$, $d_6 = \frac{e}{\sqrt{3}}$ already chosen 
for representing $S$ ({\bf Proposition \ref{Proposition 10.5}}),
leads to represent the linear mapping $A$ by the $6\times 6$ matrix ({\bf Proposition \ref{Proposition 10.7}})
$$A = \begin{bmatrix}
2 \mu I_4 & 0\\
0 & a\\
\end{bmatrix}$$
where $a$ is the $2\times 2$ matrix
$$a = \begin{bmatrix}
2 \mu & 0\\
\sqrt{3} \| h\| & 3\lambda +2\mu\\ 
\end{bmatrix} \mbox{.}$$
To express the $k^{th}$ Fitzpatrick's function we have to solve the recurrence equation 
$H_{k+1} = S-\frac{1}{4} A^* H^{-1}_k A$. We can transform this problem into two smaller problems: 
\begin{enumerate}
\item[(i)] find the solution for the $4\times 4$ spherical matrix $2 \mu I_4$,
\bigskip

\item[(ii)] find the solution for the $2\times 2$ matrix $a$. 
\end{enumerate}
Up to the multiplicative factor $2 \mu$, the first problem was already solved in 
{\bf Example \ref{Example 9.5}}, the solution is $2 \mu \alpha_k I_4 = \mu \frac{k}{k-1} I_4$.
It remains to solve the recurrence equation 
$h_{k+1} = s-\frac{1}{4} a^* h^{-1}_k a$ initialised with $h_2=s$. 
%
By induction, we observe that $h_k$ is proportional to $s$. 
Let us set $h_k = \frac{1}{2} \gamma_k s$. 
Therefore, $\gamma_k$ is satisfying the recurrence equation 
$\gamma_{k+1} = 2-\frac{1}{\cos^2 \theta}\frac{1}{\gamma_k}$ with $\gamma_2=2$. 
As one can easily check by induction, the solution is 
$\gamma_k = \frac{\sin(k\theta)}{\sin((k-1)\theta)} \frac{1}{\cos \theta}$, 
which achieves the proof. 
%
\end{proof}

%
%
\begin{remark} 
If we introduce the variable $X = \cos \theta$, then $X \gamma_k = \frac{U_{k-1}(X)}{U_{k-2}(X)}$ 
is a quotient of consecutive Chebyshev polynomials of the second kind. The above recurrence equation 
satisfied by the $\gamma_k$ and the initial value $\gamma_2=2$ follow from the recurrence relation 
and initial conditions
$$U_{0}(X) = 1 \mbox{, } U_{1}(X) = 2X \mbox{, } U_{k}(X) = 2X U_{k-1}(X)- U_{k-2}(X)$$
defining these polynomials. The well--known solution of this recurrence problem is 
$U_{k}(X) = \frac{\sin((k+1)\theta)}{\sin\theta}$, and we recover the above trigonometric expression of $\gamma_k$. 
\end{remark}

%
%
\begin{remark}
As $\theta$ approaches zero, $\gamma_k$ approaches $\frac{k}{k-1} = 2 \alpha_k$, and $h_k$ approaches 
$\alpha_k \begin{bmatrix}
2 \mu & 0\\
0 & 3\lambda +2\mu\\
\end{bmatrix}$
in accordance with {\bf Example \ref{Example 9.6}}.
\end{remark}
%

%
%
%
%
%
\section{Revising elasticity theory--Return to Hooke}

%
%
\subsection{What did Robert Hooke say?}

First, he noticed that the extension of a spring is proportional to the weight hanging on it. 
In 1660, from this experimental observation, he modeled the behavior of elastic materials by stating 
the law "UT TENSIO SIC VIS" (Latin: as the extension, so the force). He published it in the anagram form 
"ceiiinosssttuv", whose solution he gave in 1678. 

Nowadays, this law is interpreted as stating that the 
stress tensor $y$ is an isotropic linear function of the strain tensor $x$:
\begin{equation}
y=\lambda(\mbox{tr}\,x)e+2\mu x
\nonumber
\end{equation}
where $\lambda$ and $\mu$ are coefficients (Lam\'{e}'s coefficients).

\bigskip 
In the next subsection, we discuss this interpretation of Hooke's prescription. 
 
%
%
\subsection{Interpretation of Hooke's prescription}

Hooke's prescription was one--dimensional. At that time, to model the behavior of elastic bodies, 
Robert Hooke did not have at his disposal the mathematical concepts of vectors and tensors. How to interpret 
his prescription as a 6-dimensional one between strain and stress tensors? From Hooke's "UT / SIC ", 
we will retain the hypotheses of linearity and monotonicity. Why to add the condition of isotropy? 
By analogy with the behavior of springs, Robert Hooke claims that a force applied to an elastic body in some 
direction generates a proportional deformation in the same direction. In the framework of tensorial calculus, 
we will understand it as: if a direction is an eigenvector of the stress tensor, then it is an eigenvector of 
the strain tensor. Finally, we will interpret Hooke's prescription by requiring the conditions of
\begin{enumerate}
\item[--] linearity, 
\item[--] monotonicity, 
\item[--] and coaxiality,  
\end{enumerate}
for the relation between the strain tensor $x$ and the stress tensor $y$. 
According to {\bf Proposition \ref{Proposition 10.2}}, these three hypotheses 
lead to the constitutive law: 
\begin{equation}
y=\lambda(\mbox{tr}\,x)e+2\mu x + \mbox{tr}(hx)e
\nonumber
\end{equation}
where $\lambda$ and $\mu$ are scalar coefficients (Lam\'{e}'s coefficients) and $h$ is a deviatoric tensor. 
According to {\bf Proposition \ref{Proposition 10.5}} the monotonicity is insured by the inequalities: 
\begin{equation}
\mu\geq 0 \mbox{, } 3\lambda +2\mu\geq 0 \mbox{, } \mbox{tr}\left(h^2\right)\leq \displaystyle\frac{8}{3}\mu(3\lambda +2\mu) \mbox{. }
\nonumber
\end{equation}
%
%

%
%
\subsection{Four reasons for which the coaxial law reduces to the isotropic law}

%
%
\subsubsection{Classical argument: isotropy}

The isotropy assumption asks, for every rotation $R$, that the change of $x$ in $R^{-1}xR$ is 
followed by the change of $y$ in  $R^{-1}yR$. Therefore, the deviatoric tensor $h$ has to satisfy 
$R^{-1}hR=h$, and is forced to vanish. 

%
%
\subsubsection{Hermann von Helmholtz's argument: existence of a strain energy}

In the linear case, a strain energy density exists if an only if the constitutive law is symmetrical. 
Therefore, the deviatoric tensor $h$ is forced to vanish. 

%
%
\subsubsection{Lars Onsager's argument: symmetry}

As in {\bf Remark \ref{Remark 10.4}}, the symmetry of the constitutive law implies that the deviatoric tensor $h$ vanishes. 

%
%
\subsubsection{Jean--Jacques Moreau's argument: cyclic monotonicity}

The cyclic monotonicity leads to the existence of a strain energy, and therefore to the 
vanishing of the deviatoric tensor $h$ (see also {\bf Proposition \ref{Proposition 10.8}}).
\bigskip

Certainly, the reader is now aware that in our opinion, the most relevant argument is the 
fourth (cyclic monotonicity) and not the first one (isotropy). 

%
%
\subsection{Revisiting isotropic elastic materials}
In light of the previous results, we do believe that elasticity theory must be revisited. 
The constitutive laws of the so--called isotropic linear elastic materials have to be 
extended to monotone linear coaxial laws. The cyclic monotonicity is a too strong assumption, 
$k$--monotonicity up to a finite integer $n$ has to be considered. We do believe that this maximal 
integer $n$ is a relevant parameter characterizing the material, and we will say that the material 
is $n$--monotone. Classical elasticity corresponds to very large integers $n$ (therefore to very small 
angles $\theta$, cf \ref{subsec: 10.5}). 
The evaluation of the seven coefficients of the model must be performed by 
modern identification methods, allowing an optimal exploitation of the electronic recording of the 
measures (\cite{atch08,atch11}). 
Forgoing the cyclic monotonicity, the existence of an elastic 
potential $\phi(x)$ is lost. But numerical methods based on primal--dual two--field variational principles 
can be easily generalized. Simply, the integrand $\phi(x)+\phi^*(y)$ of the classical integral functional has to be 
replaced (\cite{bul12,mat11,sax102,zou07}) by the bipotential 
\begin{equation}
F_{A,n}(x,y)= \mathrm{tr}(x,y) +\frac{1}{4} \mathrm{tr} [ (y-Ax) H_{n}^{-1}(y-Ax) ]
\nonumber
\end{equation}
with
$$H_{n} = \begin{bmatrix} 
\frac{n}{n-1}\mu I_4 &0\\
0  & \frac{1}{2}\frac{\sin (n\theta)}{\sin((n-1)\theta)} \frac{1}{\cos\theta} s 
\end{bmatrix}$$
in accordance with the notation 
$$ s =\begin{bmatrix} 
2\mu & \frac{\sqrt{3}}{2}\|h\| \\ 
\frac{\sqrt{3}}{2}\|h\| & 3\lambda + 2\mu
\end{bmatrix}$$
and with the definition of the angle $\theta$ 
$$ \sqrt{2\mu\left(3\lambda + 2\mu\right)} \sin\theta = \frac{\sqrt{3}}{2}\|h\|.$$

%
%
\section{Conclusion}
\label{sec:12}

The class of $n$--monotone materials for which the constitutive law is described by a $n$--monotone 
operator is larger than the class of Generalized Standard Materials (GSM). The integer $n$ can be regarded
as a characteristic of these materials. Fitzpatrick's functions allow to construct algorithms based on 
primal--dual two--field variational principles, as easy as the GSM--specific algorithms (\cite{bul12,mat11,sax102,zou07}). 
To apply this modeling in fluid mechanics, it is necessary to replace the tensor $x$ by the strain rate tensor and the tensor 
$y$ by the Cauchy stress tensor, augmented by the spheric tensor $pe$, where $p$ is the pressure. To apply this modeling 
in thermal engineering, it is necessary to replace the tensor $x$ by the opposite temperature gradient and the tensor $y$ 
by the heat flux vector. To apply this modeling in electromagnetism, it is necessary (\cite{heh03,heh08}) to replace the tensor 
$x$ by
$$\begin{bmatrix}
0 & E^1 & E^2 & E^3\\
E^1 & 0 & -B^3 & B^2\\
E^2 & B^3  & 0 & -B^1\\
E^3 & -B^2 & B^1 & 0\\
\end{bmatrix}$$
and the tensor $y$ by 
$$\begin{bmatrix}
0 & D^1 & D^2 & D^3\\
D^1 & 0 & -H^3 & H^2\\
D^2 & H^3  & 0 & -H^1\\
D^3 & -H^2 & H^1 & 0\\
\end{bmatrix}$$
where $E$, $B$, $D$, $H$,  stand respectively for the electric field strength, the magnetic 
induction field, the electric displacement and the magnetic field intensity. 

However, in mechanical and civil engineering, the constitutive laws of many materials 
(ductile metals, metal matrix composites, wet clays, plastic soils, granular materials, etc) 
are not monotone. The class of Implicit Standard Material (ISM) is larger than the class of 
$n$--monotone materials. Every Fitzpatrick's function is globally lsc and convex, but the 
modeling of ISM only requires bipotentials which are partially lsc and convex. It would be 
interesting to generalize the concept of Fitzpatrick's sequences to the case of non--monotone 
operators. It is also expected that the recurrence formula of {\bf Proposition \ref{Proposition 7.1}} 
passes this extension with few changes. 

As a pioneering attempt for modeling Coulomb's dry friction, let us consider the following 
constitutive law in an Euclidean linear space: two vectors $x$ and $y$ have the same orientation. 
This constitutive law is not monotone, Fitzpatrick's method cannot be directly applied. Does this 
constitutive law model an IMS? Can Fitzpatrick's sequence be generalized for representing it by 
bipotentials? If we only ask for a local supremum in the {\bf Definition \ref{Definition 5.9}}  
of Fitzpatrick's function, we obtain (\cite{val09}) the extremal value
\begin{equation}
b_2(x,y)=\frac{1}{2}\langle x,y\rangle +\frac{1}{2}\| x\|\| y\|.
\nonumber
\end{equation}
This function is partially lsc and convex. Thanks to the Cauchy--Schwarz--Bunia-
kovsky inequality, 
it is a bipotential. If $\psi$ is an angle chosen between $0$ and $\pi$ such that 
$\langle x,y\rangle=\| x\|\| y\| \cos\psi$, we observe that
\begin{equation}
b_2(x,y)=\| x\|\| y\| \left( \cos \frac{\psi}{2} \right)^2. 
\nonumber
\end{equation}
A similar weakening in the definitions of Fitzpatrick's functions $F_{T,n}(x,y)$ leads (\cite{val09}) 
to the following increasing sequence of bipotentials 
\begin{equation}
b_n(x,y)=\| x\|\| y\| \left( \cos \frac{\psi}{n} \right)^n. 
\nonumber
\end{equation}
The pointwise limit
$$b(x,y)=\| x\|\| y\|$$
is a wellknown bipotential (\cite{val05}) representing the constitutive law asserting that two 
vectors $x$ and $y$ admit the same orientation; we will call it "Cauchy--Schwarz--Buniakovsky bipotential". 

We hope that this kind of extension of Fitzpatrick's sequences will prove to be very helpful to produce 
relevant bipotentials for representing the non--associated constitutive laws of the ISM evoked in 
subsection \ref{subsec: 4.3}. 

Finally, it is worth mentioning that the notion of bipotential can provide additional modeling to rate--independent phenomena 
with hysteretic behavior : linearized plasticity with hardening, finite-strain elastoplasticity with nonlinear and non--associated kinematic hardening rules,
damage in ductile materials, phase transformations 
in shape memory alloys, delamination, ferromagnetism, superconductivity (\cite{boss94,boss00,mas97,pri96}), 
quasistatic evolution of fractures, and crack propagation in brittle materials.
%
%
%
%
%

\bigskip
{\bf Acknowledgements} Thanks are due to our long--term discussion partners {\em Ma-rius Buliga} (Institute of Mathematics of the Romanian 
Academy, Bucarest, Romania) and {\em G\'{e}ry de Saxc\'{e}} (University of Lille, France), for numerous stimulating 
discussions on the subject matter of this paper and for making a number of perceptive comments. 
The first--named author is grateful for partial financial support by the European CNRS Franco--Romanian Associated 
Laboratories agreement (LEA Math--Mode, Math\'{e}matiques et Mod\'{e}lisation) established between the CNRS and the Romanian 
Academy. Parts of this work were done when he has been invited at the Simion Stoilow Institute of the Romanian Academy. 
Moreover, he is indebted to {\em Jean--Marie Souriau} from the University of Aix--Marseille and {\em Jean--Jacques Moreau} 
from the University of Montpellier for being source of inspiration and encouragement. 
 
%
%
%
%
%
%

%

\end{document}